\documentclass[10pt,a4paper]{amsart}
\usepackage{amsmath,amsthm,amssymb,amscd,mathrsfs}
\usepackage[all]{xy}
\usepackage{tikz}
\usepackage{txfonts}
\usepackage{hyperref}
\hypersetup{colorlinks=true, linkcolor=black, citecolor=black}
\newtheorem{lemme}{Lemme}[section]

\newtheorem*{conjecture*}{Conjecture}

\newtheorem*{theorem*}{Theorem}
\newtheorem*{corollary*}{Corollary}
\newtheorem{lemma}[lemme]{Lemma}
\newtheorem{proposition}[lemme]{Proposition}
\newtheorem{remark}[lemme]{Remark}

\newtheorem{theorem}[lemme]{Theorem}

\newtheorem{corollary}[lemme]{Corollary}

\baselineskip 18pt \textwidth 16cm  \theoremstyle{plain}

\newcommand{\Aut}{\operatorname{Aut}}
\newcommand{\End}{\operatorname{End}}
\newcommand{\Hom}{\operatorname{Hom}}
\newcommand{\Isom}{\operatorname{Isom}}
\renewcommand{\Im}{\operatorname{Im}}
\newcommand{\Ker}{\operatorname{Ker}}

\newcommand{\Ind}{\operatorname{Ind}}

\newcommand{\Res}{\operatorname{Res}}

\newcommand{\C}{\mathbb C}

\newcommand{\Gal}{\operatorname{Gal}}

\newcommand{\SL}{\operatorname{SL}}

\newcommand{\Sp}{\operatorname{Sp}}

\newcommand{\nnn}{\operatorname{N}}

\newcommand{\tr}{\operatorname{tr}}

\begin{document}

\title{Weil representations over finite fields and Shintani lift }
\date{March 2013}
\keywords{Weil representation, base change}

\author{Guy Henniart}
\address{Institut Universitaire de France et Universit\'e de Paris-Sud\\ Laboratoire de Math\'ematiques d'Orsay\\
Orsay Cedex, F-91405; CNRS, Orsay cedex\\ F-91405.}
\email{Guy.Henniart@math.u-psud.fr}

\author{  Chun-Hui Wang}
\address{NCMIS, Academy of Mathematics and Systems Science\\
Chinese Academy of Sciences\\
Beijing, 100190, P.R. China}
\email{cwang@amss.ac.cn}

\subjclass[2000]{11F27, 20C33  (Primary)}

\maketitle
\setcounter{tocdepth}{2}
\begin{abstract}
Let $\Sp_V(F)$ be the group of isometries of a symplectic vector space $V$ over a finite field $F$ of odd cardinality. The group $\Sp_V(F)$ possesses distinguished representations--- the Weil representations. We know that they are compatible with base change in the sense of Shintani for a finite extension $F'/F$. The result is also true for the group of similitudes of $V$.
\end{abstract}
\section{Introduction}\label{I}
    Let $F$ be a finite field of odd cardinality $q$, and let $\psi$ be a non-trivial character of $F$. Consider a symplectic vector space $V$ over $F$ of  finite dimension and write $\textbf{Sp}_V$ for its group of isometries seen as an algebraic group over $F$. To $\psi$ is attached a canonical class of representations of $\textbf{Sp}_V(F)$, the Weil representations $W_{\psi}$\cite{Gero}. Let $F'$ be a finite extension of $F$ with  Frobenius automorphism $\sigma$. In this paper, we establish the behavior of $W_{\psi}$ with respect to Shintani lifting   from $F$ to $F'$. We recall that there is a norm map $N$ yielding a bijection from the set of $\textbf{Sp}_V(F')$-conjugacy classes of $\sigma \ltimes \textbf{Sp}_V(F')$, a subset  of $\Gal(F'/F)\ltimes \textbf{Sp}_V(F')$,  onto the set of conjugacy classes of $\textbf{Sp}_V(F)$.  Now set $\psi'=\psi\circ \tr_{F'/F}$.
\begin{theorem*}
There is  a canonical  extension $\widetilde{W_{\psi'}}$  of $W_{\psi'}$ to $\Gal(F'/F)\ltimes \textbf{Sp}_V(F')$ such that
 $$(\star) \quad \quad \tr \widetilde{W_{\psi'}}(\sigma, g)=\tr W_{\psi}\big( Ng\big)$$
 for any $g\in \textbf{Sp}_V(F')$.
\end{theorem*}
We actually give an explicit model for $\widetilde{W_{\psi'}}$, using the Schr\"odinger model of $W_{\psi'}$(cf. \S \ref{EWR}). Note that our results are in fact more general, in that we consider norm maps for any power of $\sigma$: the corresponding  statement is in \S \ref{EWR}.   We also establish the analogous results  in \S \ref{EWR} for the Weil representation of the group $\textbf{GSp}_V$ of similitudes of $(V, \langle, \rangle)$---the class of that representation does not depend on the choice of $ \psi$. With the same methods, we can prove that the Weil representations of general linear
 groups and unitary groups defined by G\'erardin in [\cite{Gero}, \S 2 and \S 3] are compatible with Shintani lifting as well. We shall come back to those cases, with applications, in future work. \\

In fact the character relation ($\star$) in the theorem is valid for a pair $(\sigma, g)$, where $g$ is in the semi-direct product $\textbf{Sp}_V(F') \ltimes \textbf{H}_V(F')$. But the identity is $0=0$ unless $(\sigma, g)$ is conjugate to $(\sigma, g')$ with $g'$ in $\textbf{Sp}_V(F') \times \textbf{Z}_V(F')$,  $\textbf{Z}_V$ being the centre of $\textbf{H}_V$ (see \S \ref{SOTCOFTEWR}). So in effect, we are reduced to proving ($\star$) for a fixed $g$ in $\textbf{Sp}_{V}(F')$, or more conveniently for a fixed norm $h$ in $\textbf{Sp}_V(F)$.\\

We proceed by induction on $2n=\dim V$, allowing the field $F$ to vary. If $h$ belongs to some proper  parabolic subgroup of $\textbf{Sp}_V(F)$, we use the mixed Schr\"odinger model (\S \ref{EWR}) to reduce to a smaller dimension (\S  \ref{reduction}). If $h$ stabilizes a decomposition $V= V_1 \oplus V_2$ of $V$ into  a direct sum of two non-zero symplectic subspaces, again we are reduced to a smaller dimension (\S \ref{OD}, \S \ref{reduction}). The remaining case is when $h$ is a regular element of a maximally elliptic torus $\textbf{T}(F)$ of $\textbf{Sp}_V(F)$. More concretely, $V$ is a one-dimensional skew-hermitian vector space over a finite extension $E$ of $F$ of degree $n$ and $h$ is an element in $E^{\times}$ (acting on $V$), of norm $1$ in the subfield $E_+$ such that $[E:E_+]=2$. That case is treated in \S \ref{TCOS} and \S \ref{EOP} with some explicit computations.\\

\emph{Acknowledgements.}
G.Henniart would like to thank C.Bonnaf\'e, F. Digne and J. Michel for mathematical discussions on the present topic. This paper is an extension of the last chapter of the Ph.D. thesis of C-H. Wang, written with G.Henniart as advisor.

\emph{Plan}
\begin{itemize}
\item[1]. Introduction
\item[2]. Notation
\item[3]. Norm maps
\item[4]. Extended Weil representation
\item[5]. Support of the character of the extended Weil representation
\item[6]. Orthogonal decomposition
\item[7]. Restrictions to a parabolic subgroup
\item[8]. Reductions
\item[9]. The case of $\SL_2(F)$
\item[10]. End of  proof.
\end{itemize}
\section{Notation}\label{N}
Throughout the paper, $F$ is a finite field of \textbf{odd} cardinality $q$, and $F'$ is a finite extension of $F$ of degree $m$; we put $\Gamma=\Gal(F'/F)$.  We also fix an algebraic closure $\overline{F}$ of $F'$,  and write $\sigma$ for the Frobenius automorphism $x \longmapsto x^q$ of $\overline{F}$; it restricts to the Frobenius automorphism of $F'$ over $F$, for which we also write $\sigma$. For any positive integer $d$, we let $F_d$ be the degree $d$ extension of $F$ in $\overline{F}$; thus $F'= F_m$.

 If $\sigma$ acts on a set $X$, we write $X_{\sigma}$ for the set of fixed points of $\sigma$ in $X$, we use a similar notation for powers of $\sigma$. If $G$ is a group, we write  $\mathcal{C}(G)$ for  the vector space of complex valued class functions on $G$;  if $ G\ltimes H$ is a semi-direct product group, then the law will be given by $(g, h) \cdot (g', h')=(gg', h\cdot g(h'))$, where $g(h')$ denotes the action of $g\in G$ on the element $h'$ of the invariant subgroup $H$.

\section{Norm maps}\label{NM}
Let $i$ be an integer, $0 \leq i \leq m-1$, write $d$ for the  greatest common divisor of $m$ and $i$; put $i=dj$, $m=d\mu$ for some integers $j$ and $\mu$. We choose an integer $t$ such that $ti\equiv d \,(mod \, m)$.

Let $\textbf{G}$ be a connected linear algebraic group over the field $F$. We consider the semi-direct product $\Gal(\overline{F}/F) \ltimes \textbf{G}(\overline{F})$. In \cite{Gyoja}, Gyoja constructs  a norm map $\nnn_{i,t}$ from $\sigma^i \ltimes\textbf{G}(F') $ to $\textbf{G}(F_d)$ in the  following way:

For $g$ in $\textbf{G}(F')$, choose $\alpha=\alpha(g) $ in $\textbf{G}(\overline{F})$ such that
$$ \big(1, \alpha^{-1}\sigma^d (\alpha)\big)=(\sigma^{-it},1) \cdot (\sigma^i,g)^t$$

and let
$$ \nnn_{i,t}(\sigma^i, g)=\alpha\Big(g  \sigma^i(g)\cdots \sigma^{i\big(\mu-1\big)}(g)\Big) \alpha^{-1}.$$
That element $\nnn_{i,t}(\sigma^i, d)$ does belong to $\textbf{G}(F_d)$, and its conjugacy class in $\textbf{G}(F_d)$ does not depend on the choice of $\alpha$. Moreover, Gyoja shows that $\nnn_{i,t}$ induces a bijection from the set of $\textbf{G}(F')$-conjugacy classes in $\sigma^i\ltimes \textbf{G}(F')$ onto the set of conjugacy classes in $\textbf{G}(F_d)$. It is immediate that this bijection is $\sigma$-equivariant.\\
\ \\
Remarks:

(i) For $i=t=1$, we recover the classical Shintani norm map[\cite{Dign}, \cite{Shint}]. Note that $\nnn_{i,t}$ does depend on the choice of $t$; for  instance, it can be proved that $\nnn_{1, m+1}=Sh_{F/F} \circ \nnn_{1,1}$, where $Sh_{F/F}$ is the notation for the Shintani self-lift of \cite{Dign}.

(ii) Putting $\tau=\sigma^d$, which is the Frobenius automorphism for $\overline{F}/F_d$, we see that for $\sigma^i(g)=\tau^j(g)$, the norm $\nnn_{i,t}(\sigma^i, g)$ is the same as $\nnn_{j,t}(\tau^j, g)$, thus we can always reduce our considerations to the case where $i$ is prime to $m$, at the cost of allowing a change of base field from $F$ to $F_d$.

(iii) Assume that $\textbf{G}$ is commutative; then for $g$ in $\textbf{G}(F')$, we have $\nnn_{i,t}(\sigma^i, g)=g\sigma^i(g) \cdots \sigma^{i(\mu-1)}(g)$, in other words, this is simply the usual norm of $g$ from $\textbf{G}(F')$ to $\textbf{G}(F_d)$.\\

Composing with $\nnn_{i,t}$ gives a vector space isomorphism $\mathcal{N}_{i,t}$ of $\mathcal{C}(\textbf{G}(F_d))$ onto the vector space $\mathcal{C}(\sigma^i \ltimes \textbf{G}(F'))$ of complex valued functions on $\sigma^i \ltimes \textbf{G}(F')$ which are invariant under conjugation by $\textbf{G}(F')$. It induces an isomorphism of $\mathcal{C}(\textbf{G}(F_d))_{\sigma}$ onto the vector space $\mathcal{C}(\sigma^i \ltimes \textbf{G}(F'))_{\sigma}$ of complex valued functions on $\sigma^i \ltimes \textbf{G}(F')$ which are invariant under conjugation by $\Gamma \ltimes \textbf{G}(F')$.

When $\textbf{G}$ is abelian, and $\chi$ is a character of $\textbf{G}(F)$, we get a character $\widetilde{\chi'}$ of $\Gamma\ltimes \textbf{G}(F')$ by composing $\chi$ with the usual norm $\nnn$ from $\textbf{G}(F')$ to $\textbf{G}(F)$ and extending trivially on $\Gamma$. So for $g\in \textbf{G}(F')$,    we have
 $$ \widetilde{\chi'} (\sigma, g) =\chi (\nnn(g)).$$

 When $\textbf{G}$ is non-abelian, and $\chi$ is the character of an irreducible representation of $\textbf{G}(F)$, it is \emph{not} generally the case that $\mathcal{N}_{i,t} (\chi) $ is the restriction to $\sigma^i \ltimes \textbf{G}(F')$ of some character  of  a representation of $\Gamma \ltimes \textbf{G}(F')$. However this paper is concerned with a situation where it is indeed the case.

For later use, we shall recall some results of Gyoja in \cite{Gyoja}:
\begin{lemma}\label{lemmaofGyjo}

(i) For any $\chi, \chi'\in \mathcal {C}\big(\textbf{G}(F_d)\big)$, we have $\langle \chi, \chi'\rangle= \langle \mathcal{N}_{i,t}( \chi), \mathcal{N}_{i,t}(\chi')\rangle$, where $\langle \chi, \chi'\rangle:= \frac{1}{|\textbf{G}(F_d)|} \sum_{x\in \textbf{G}(F_d)} \chi(x) \overline{\chi'(x)}$ and $\langle \mathcal{N}_{i,t}( \chi), \mathcal{N}_{i,t}(\chi')\rangle:=\frac{1}{|\sigma^i\ltimes \textbf{G}(F')|} \sum_{y\in \textbf{G}(F')} \chi\big( \nnn_{i,t}(\sigma^i, y)\big) \overline{\chi'\big(\nnn_{i,y}(\sigma^i, y)\big)}$.

(ii)  Through  the   lifting maps $\mathcal{N}_{i,t}$ by allowing  $i$ to vary from $0$ to $m-1$,  we can  decompose
$\mathcal {C}\big(\Gal(F'/F)\textbf{G}(F')\big) $ as  the direct sum $\oplus_{i=0}^{m-1} \mathcal{C}\big(\textbf{G}(F')_{\sigma^i}\big)_{\sigma}$.

(iii) The above decomposition  is compatible  with the usual  induction map, the restriction map, the product map, etc. For example, if $\textbf{H}$ is a connected algebraic  subgroup of $\textbf{G}$ defined over $F$, then for $\chi\in \mathcal {C}\big(\Gal(F'/F)\textbf{G}(F')\big) $ such that $\chi|_{\sigma^i \ltimes \textbf{G}(F') }= \mathcal{N}_{i,t} (\chi')$ for some $\chi'\in \mathcal{C}\big(\textbf{G}(F_d)\big)_{\sigma}$,  we have $\Res_{\Gal(F'/F)\textbf{H}(F')}^{\Gal(F'/F)\textbf{G}(F')}(\chi)= \mathcal{N}_{i,t} \circ \Res_{\textbf{H}(F_d)}^{\textbf{G}(F_d)}(\chi')$.

\end{lemma}

\section{Extended Weil representation}\label{EWR}
As in the introduction, we fix a symplectic vector space $V$ over $F$, and write $2n$ for its dimension, $\langle, \rangle$ for the symplectic form on $V$. We see $V$ as a linear algebraic group,  denoted by the bold letter $\textbf{V}$, and similarly for the group $\textbf{Sp}_V$ of isometries of $\textbf{V}$, the group $\textbf{GSp}_V$ of similitudes of $V$.

Let $\textbf{H}_V$ be the Heisenberg group over $F$ associated to $V$: for each $F$-algebra $R$, $\textbf{H}_V(R)$ is the set $\textbf{V}(R) \oplus R$, endowed with the group law
$$(v_1, t_1) (v_2, t_2)=(v_1 + v_2, t_1 + t_2 + \frac{1}{2} \langle v_1, v_2 \rangle_R),$$
where the form $\langle, \rangle_R$ is obtained by scalar extension. Then $\textbf{H}_V$ is a non-abelian connected algebraic group over $F$, with centre $\textbf{Z}_V$ such that $\textbf{Z}_V(R)=\{ (0, x)| x\in R\}$.

Fix a non-trivial character $\psi$ of $F$, and put $\psi'= \psi \circ \tr_{F'/F}$. To $\psi$ is associated the Weil representation of $\textbf{Sp}_V(F) \textbf{H}_V(F)$--- it is in fact an isomorphism class of representations. We write $\rho$ for that Weil representation, and $\tr(\rho)$ for  its character. Similarly to $\psi'$ is associated the Weil representation $\rho'$ of $\textbf{Sp}_V(F') \textbf{H}_V(F')$. Indeed for each positive integer $d$, we have a Weil representation $\rho_d$ of $\textbf{Sp}_V(F_d)\textbf{H}_V(F_d)$ associated to the character $\psi_d=\psi\circ \tr_{F_d/F}$. Our main result is the following:
\begin{theorem}\label{shintaniliftingforSp}
There is a unique extension $\widetilde{\rho'}$ of $\rho'$ to $\Gamma \ltimes \textbf{Sp}_V(F') \textbf{H}_V(F')$ such that, for integers $i,t,d$ as in \S \ref{NM}, and $g \in \textbf{Sp}_V(F')\textbf{H}_V(F')$, we have
$$(\star) \quad \quad \tr\widetilde{\rho'}(\sigma^i, g)= \tr\rho_d \big( \nnn_{i,t}(\sigma^i, g)\big).$$
In particular, for $i=t=1$, we obtain
$$\tr\widetilde{\rho'}(\sigma, g)= \tr\rho\big( \nnn_{1,1}(\sigma, g)\big),$$
i.e. the Weil representation is ``compatible'' with Shintani lifting.
\end{theorem}

As indicated in the introduction, this will be proved progressively. In this \S \ref{EWR}, we use the Schr\"odinger model of $\rho'$ to construct an extension $\widetilde{\rho'}$  such that $\tr\widetilde{\rho'}(\sigma)=\tr\rho(1)$; note that $\nnn_{1,t}(\sigma)=1$ for all possible $t$'s. Then by Clifford theory, $\widetilde{\rho'}$ is the unique extension  satisfying  this simple character relation, so the remaining problem will be to prove ($\star$) in general. For this purpose, in the following section \S \ref{SOTCOFTEWR}, we  examine the support of the character $\widetilde{\rho'}$; in \S  \ref{OD} and \S \ref{RTAPPS}, we consider  the restriction of $\widetilde{\rho'}$ to some interesting subgroups; the proof of ($\star$) will be reduced to  a very special case, and we treat  this special case   in \S \ref{TCOS} and \S \ref{EOP}. \\

Firstly admitting the theorem, let us derive a consequence for the groups of symplectic similitudes.  Put  $\pi=\Ind_{\textbf{Sp}_V(F)}^{\textbf{GSp}_V(F)}\rho|_{\textbf{Sp}_V(F)}$;  it is the Weil representation of $\textbf{GSp}_V(F)$ which is  independent( up to isomorphism ) of the choice of $\psi$\cite{Gero};    similarly as in \S \ref{EWR},  for each factor  $d$  of $m$,   we denote  the corresponding  Weil representation  of $\textbf{GSp}_V(F_d)$  by $\pi_d$, and   also  write $\pi'$ for $d=m$.

\begin{theorem}
There is a unique extension $\widetilde{\pi'}$ of $\pi'$ to $\Gamma \ltimes \textbf{GSp}_V(F') $ such that, for integers $i,t,d$ as in \S \ref{NM}, and $g' \in \textbf{GSp}_V(F')$, we have
$$(\star') \quad \quad \tr\widetilde{\pi'}(\sigma^i, g')= \tr\pi_d \big( \nnn_{i,t}(\sigma^i, g')\big).$$
Moreover, the induced representation of $\Gamma \ltimes \textbf{GSp}_V(F')$ from  the representation  $\widetilde{\rho'}$ of $\Gamma \ltimes \textbf{Sp}_V(F')$ satisfies the desired conditions.
\end{theorem}
\begin{proof}
Uniqueness comes from Lemma \ref{lemmaofGyjo} (ii),  and by (iii) in the same lemma and Theorem \ref{shintaniliftingforSp},  we see $[\Ind_{\Gal(F'/F) \textbf{Sp}_V(F')}^{\Gal(F'/F)\textbf{GSp}_V(F')} \Big( \tr \widetilde{\rho'}\Big)](\sigma^i, g')= [\mathcal{N}_{i,t} \Big( \Ind_{\textbf{Sp}(F_d)}^{\textbf{GSp}(F_d)} \tr \rho_d \Big)](\sigma^i, g') $; in this equality, the second term is equal to $\tr \pi_d \big(\nnn_{i,t}(\sigma^i, g')\big)$, so the results follow.
\end{proof}

Now let us fix a complete polarisation $V= X\oplus X^{\star}$ of $V$, so that $X, X^{\star}$ are two Lagrangian subspaces of $V$. We denote the corresponding algebraic groups over $F$ by $\textbf{X}, \textbf{X}^{\star}$ and $\textbf{V}$ respectively and  write $\epsilon'$  for the unique non trivial quadratic character of $F'^{\times}$. Then the Weil representation $\rho'$ of $\textbf{Sp}_V(F') \textbf{H}_V(F')$ can be realized in the space $\C[\textbf{X}^{\star}(F')]$ of complex functions on $\textbf{X}^{\star}(F')$ by the following formulas[cf. \cite{Gero}]:
\begin{equation}\label{representationsp1}
\rho'\big(1,(x+x^{\star}+k)\big)f(y^{\star})=\psi'(k+\langle y^{\star},x\rangle) f(x^{\star}+y^{\star}),
\end{equation}
\begin{equation}\label{representationsp2}
\rho'\big(\begin{pmatrix}
  1&b\\
  0 & 1
\end{pmatrix},1\big)f(y^{\star})=\psi'(\frac{\langle by^{\star},y^{\star}\rangle}{2}) f(y^{\star}),
\end{equation}
\begin{equation}\label{representationsp3}
\rho'\big(\begin{pmatrix}
  a& 0\\
  0 &a^{\star -1 }
\end{pmatrix},1\big)f(y^{\star})=\epsilon'(\det(a)) f(a^{\star}y^{\star}),
\end{equation}
\begin{equation}\label{representationsp4}
\rho'\big(\begin{pmatrix}
  0&c'\\
  c &0
\end{pmatrix},1\big)f(y^{\star})=\gammaup(\psi')^{-n}\epsilon'(\det(c))\int_{\textbf{X}^{\star}(F')}f(x^{\star})\psi'(\langle x^{\star},c^{-1}y^{\star}\rangle) dx^{\star},
\end{equation}
where $b\in \Hom(\textbf{X}^{\star}(F'), \textbf{X}(F'))$, $a\in \Aut(\textbf{X}(F'))$, and $a^{\star} \in \Aut(\textbf{X}^{\star}(F'))$ is the adjoint of $a$ with respect to the bilinear form $\textbf{X}(F') \times \textbf{X}^{\star}(F') \longrightarrow F'$ given by $(x, x^{\star}) \longmapsto \langle x, x^{\star} \rangle$, and finally $c\in \Isom(\textbf{X}^{\star}(F'),\textbf{X}(F'))$, $c'\in \Isom(\textbf{X}(F'), \textbf{X}^{\star}(F'))$.\\

Let $I_{\sigma}$ be the automorphism of $\C[\textbf{X}^{\star}(F')]$ given by
$$I_{\sigma}(f) (x) =f(\sigma^{-1}(x)) \textrm{ for } f\in \C[\textbf{X}^{\star}(F')], x\in \textbf{X}^{\star}(F').$$
It is easily verified on the formulas (\ref{representationsp1}) to (\ref{representationsp4}) that
$$I_{\sigma}\rho'(g)I_{\sigma}^{-1}=\rho'(\sigma (g)) \textrm{ for } g \in \textbf{Sp}_V(F')\textbf{H}_V(F');$$
in formulas (1) and (2), one uses the facts that $\psi'$ is $\sigma$-invariant and that the symplectic form on $\textbf{V}(F')$ is $\sigma$-equivariant; in formulas (3) and (4), one uses moreover that $\epsilon'$ is also $\sigma$-invariant.

Since $(I_{\sigma}^m)$ is the identity, it follows that there is a unique extension of the action $\rho'$ of $\textbf{Sp}_V(F')\textbf{H}_V(F')$ on $\C[\textbf{X}^{\star}(F')]$ to an action $\widetilde{\rho'}$ of $\Gamma\ltimes \textbf{Sp}_V(F') \textbf{H}_V(F')$ such that $\sigma$ acts via $I_{\sigma}$.  By the formulas for $I_{\sigma}$, $\tr\widetilde{\rho'}(\sigma) = q^n$ which is also $\tr\rho(1)$.\\

\section{Support of the character of the extended Weil representation}\label{SOTCOFTEWR}
It is a result of [\cite{Gero}, p.84-85] that $\check{\rho} \otimes \rho$ is isomorphic to the representation of $\textbf{Sp}_V(F) \textbf{H}_V(F)$ induced from the trivial representation of $\textbf{Sp}_V(F) \textbf{Z}_V(F)$;  in particular, the character of $\rho$ is $0$ outside the conjugates of $\textbf{Sp}_V(F) \textbf{Z}_V(F)$. We establish the analogous fact for $\widetilde{\rho'}$.
\begin{proposition}
$\check{\widetilde{\rho'}} \otimes \widetilde{\rho'}$ is isomorphic to the representation of $\Gamma \ltimes \textbf{Sp}_V(F') \textbf{H}_V(F')$ induced from the trivial character of $\Gamma\ltimes \textbf{Sp}_V(F') \textbf{Z}_V(F')$.
\end{proposition}
\begin{proof}
The representation $\widetilde{\lambda'}=\Ind_{\Gal(F'/F)\ltimes \textbf{Sp}_V(F')\textbf{Z}_V(F')}^{\Gal(F'/F)\ltimes \textbf{Sp}_V(F')\textbf{H}_V(F')} 1$ can  be realized in $\C[\textbf{V}(F')]$ by the following formulas:
\begin{equation}\label{equationH}
\widetilde{\lambda'}(h)(F)(v)=F(v+v_0) \textrm{ for }  h \in \textbf{H}_V(F') \textrm{ with projection } v_0 \textrm{ on } \textbf{V}(F'),
\end{equation}
\begin{equation}\label{equationSp}
\widetilde{\lambda'}(s)(F)(v)=F(s^{-1}v) \textrm{ for } s \in \textbf{Sp}_V(F'),
\end{equation}
\begin{equation}\label{equationSigma}
\widetilde{\lambda'} (\sigma)(F)(v)=F(\sigma^{-1}(v)).
\end{equation}
Recall, for $g\in \textbf{Sp}_V(F')\textbf{H}_V(F')$, we have
$$\widetilde{\rho'}(g)\widetilde{\rho'}(\sigma) (f)(x^{\star})=  \rho'(g) I_{\sigma}(f) (x^{\star}), $$
where $x^{\star} \in \textbf{X}^{\star}(F')$. As shown in \cite{Gero}, $\check{\widetilde{\rho'}}|_{\textbf{Sp}_V(F')\textbf{H}_V(F')}$ is isomorphic to the Weil representation ${\rho'}^-$ associated to the character  $\psi^-$ ( defined as $x \longrightarrow  \psi(-x)$) of $F$. Hence  the extended   representation $\check{\widetilde{\rho'}}$ can be realized in $\C[\textbf{X}(F')]$  by the analogous formula:
$$\check{\widetilde{\rho'}}(g) \widetilde{\rho'}(\sigma)(f)(x)= e^{\frac{2\pi i k'}{m}}{\rho'}^- (g) I_{\sigma}(f) (x)$$
for $ g\in \textbf{Sp}_V(F')\textbf{H}_V(F')$, $x \in \textbf{X}(F')$ and  $m$-th root of unity $\xi$. Computing its trace at $\sigma$, we  get   $\xi=1$.

Now let $I$ be an automorphism on $\C[\textbf{V}(F')]$ defined  by $I(f)(x+x^{\star})=\psi'\big(\langle x, x^{\star}\rangle) f(x+x^{\star})$ for $x\in \textbf{X}(F')$ and $x^{\star}\in \textbf{X}^{\star}(F')$.  In  [\cite{Gero}, p. 84], G\'erardin verifies that $I\cdot\widetilde{\lambda'}(h)= \check{ \rho'} \otimes \rho'(h) \cdot I$ for $h \in \textbf{H}_V(F')$.  Moreover, G\'erardin observes that any other such endomorphism $I'$ is the composition of a convolution operator $\phi \star$ on $\C[\textbf{V}(F')]$ with $I$. If one takes this $\phi$:
$$ \phi(x+x^{\star}):= \psi'(2 \langle x^{\star}, x\rangle) \textrm{ for } x\in \textbf{X}(F'), x^{\star} \in \textbf{X}^{\star}(F'),$$
then the results in [\cite{Gero}, p.85] say that
$$I' \widetilde{\lambda'}(s)={\rho'}^- \otimes \rho'(s) I'  \textrm{ for } s\in \textbf{Sp}_V(F').$$
Moreover, by definition, we see $I'\circ I_{\sigma}=I_{\sigma}\circ I'$, so $I'\cdot\widetilde{ \lambda'}(\sigma)f=I'I_{\sigma} (f)=I_{\sigma} I'(f)={\rho'}^-\otimes \rho'(\sigma) I'(f)$ for $f\in \textbf{V}(F')$,  and  the result follows.
\end{proof}
\begin{corollary}
$\tr\widetilde{\rho'}$ is $0$ outside the conjugates of $\Gamma \ltimes \textbf{Sp}_V(F') \textbf{Z}_V(F')$.
\end{corollary}
\section{Orthogonal decomposition}\label{OD}
Let $V=V_1 \oplus V_2$ be a decomposition of $V$ into the direct orthogonal sum of two symplectic spaces $V_1$ and $V_2$. We then have a group homomorphism $$\textbf{H}_{V_1}(F) \times \textbf{H}_{V_2}(F) \longrightarrow \textbf{H}_V(F)$$
$$[(v_1, k_1), (v_1, k_2)] \longmapsto (v_1 + v_2, k_1 +k_2),$$
and an obvious embedding
$$\textbf{Sp}_{V_1}(F) \times \textbf{Sp}_{V_2}(F) \longrightarrow \textbf{Sp}_V(F),$$
so we get a group homomorphism
$$(\textbf{Sp}_{V_1}(F) \textbf{H}_{V_1}(F) ) \times (\textbf{Sp}_{V_2}(F) \textbf{H}_{V_2}(F) ) \stackrel{\delta}{\longrightarrow} \textbf{Sp}_V(F) \textbf{H}_V(F).$$
It is a result of \cite{Gero} that $\rho\circ \delta$ is isomorphic to the (external) tensor product of the Weil representations $\rho_1$, $\rho_2$ associated to $\psi$ and the symplectic spaces $V_1$, $V_2$.

Over $F'$, we have analogously a group homomorphism
$$\textbf{Sp}_{V_1}(F') \textbf{H}_{V_1}(F') \times \textbf{Sp}_{V_1}(F') \textbf{H}_{V_2}(F') \longrightarrow \textbf{Sp}_{V}(F') \textbf{H}_V(F').$$
It clearly extends to a group homomorphism
$$\delta': \Gamma \ltimes [\textbf{Sp}_{V_1}(F') \textbf{H}_{V_1}(F') \times \textbf{Sp}_{V_2}(F') \textbf{H}_{V_2}(F') ] \longrightarrow \Gamma \ltimes \textbf{Sp}_{V}(F') \textbf{H}_{V}(F'),$$
and the left hand side is a subgroup of $(\Gamma \ltimes \textbf{Sp}_{V_1}(F') \textbf{H}_{V_1}(F') ) \times (\Gamma \ltimes \textbf{Sp}_{V_2}(F') \textbf{H}_{V_2}(F') ) $. We write $\widetilde{\rho'_1}$, $\widetilde{\rho_2'}$ for the extended Weil representations of the two components of that groups.
\begin{proposition}\label{orthognoal}
The representation $\widetilde{\rho'} \circ \delta'$ is isomorphic to the restriction  of $\widetilde{\rho_1'} \otimes \widetilde{\rho_2'}$ to $\Gamma \ltimes [(\textbf{Sp}_{V_1}(F') \textbf{H}_{V_1}(F') ) \times (\textbf{Sp}_{V_2}(F') \textbf{H}_{V_2}(F'))]$.
\end{proposition}
\begin{proof}
On restriction to $(\textbf{Sp}_{V_1}(F') \textbf{H}_{V_1}(F') ) \times (\textbf{Sp}_{V_2}(F') \textbf{H}_{V_2}(F'))$, that is the  above mentioned result of G\'erardin. To compare
the two extensions to the semi-direct product with $\Gamma$, it is enough to compare the traces at $\sigma$(provided they are non-zero). If $\dim V_1=2n_1$, and $\dim V_2=2n_2$, the trace of $\widetilde{\rho'} \circ \delta'$ at $\sigma$ is $q^n$, and the trace of $\widetilde{\rho'_1} \otimes \widetilde{\rho'_2}$ at $(\sigma, \sigma)$ is $q^{n_1} q^{n_2}=q^n$.
\end{proof}
\section{Restriction to a proper parabolic subgroup}\label{RTAPPS}
Let now $V_+$ be a non-trivial isotropic subspace of $V$, and $V_0$ the symplectic space $V_+^{\perp}/{V_+}$. Write $P$ for parabolic subgroup of $Sp_V$ which is the stabilizer of $V_+$.  Write the corresponding linear algebraic groups over $F$ as $\textbf{V}_+$, $\textbf{V}$, $\textbf{V}_0$ $\textbf{P}$, $\textbf{Sp}_V$. Then  we have an exact sequence  of algebraic groups
$$1 \longrightarrow \textbf{U} \longrightarrow \textbf{P}  \longrightarrow \textbf{GL}_{V_+} \times \textbf{Sp}_{V_0} \longrightarrow 1,$$
where $\textbf{U}$ is the unipotent radical of $\textbf{P}$ and the homomorphisms $\textbf{P} \longrightarrow \textbf{GL}_{V_+}$, $\textbf{P} \longrightarrow \textbf{Sp}_{V_0}$ are given by the induced actions on $\textbf{V}_+$, $\textbf{V}_+^{\perp}$. Note that when $V_+$ is a maximal isotropic subspace, $\textbf{V}_0=\{ 0\}$ and $\textbf{Sp}_{V_0}$ is just the trivial group.

 Let $H_{\perp}$ be the inverse image of $V^{\perp}_+$ in $H_V$, then the Heisenberg group $H_{V_0}$ appears  as the quotient of $H_{\perp}$ by the subgroup $V_+ \oplus 0$. Of course, $P$ stabilizes $H_{\perp}$. All are viewed as algebraic groups over $F$, and denoted by the bold letters. It follows that  we have a natural homomorphism
$$\textbf{P} \,\textbf{H}_{\perp} \stackrel{\textbf{$\eta$}}{\longrightarrow} \textbf{GL}_{V_+} \times (\textbf{Sp}_{V_0} \textbf{H}_{V_0}).$$

Writing $\rho_0$ for the Weil representation of $\textbf{Sp}_{V_0}(F) \textbf{H}_{V_0}(F)$ associated to $\psi$( the trivial representation of the trivial group if $\textbf{V}_0(F)=0$), and $\epsilon$ for the unique non-trivial quadratic character $F^{\times}$, it is a result of G\'erardin[\cite{Gero}, Theorem 2.4 ] that the restriction of $\rho$ to $\textbf{P}(F) \textbf{H}_{V}(F)$ is induced from the representation of $\textbf{P}(F)\textbf{H}_{\perp}(F)$ obtained by (composing with $\eta$ over $F$) the representation given by  $\epsilon\circ \det$  on $\textbf{GL}_{V_+}(F)$ and by $\rho_0$ on $\textbf{Sp}_{V_0}(F) \textbf{H}_{V_0}(F)$.

Now \textbf{$\eta$} over $F'$ readily  extends to a homomorphism
$$\eta': \Gamma \ltimes \Big( \textbf{P}(F') \textbf{H}_{ \perp}(F') \Big)  \longrightarrow \Gamma \ltimes \Big( \textbf{GL}_{V_+}(F') \textbf{Sp}_{V_0}(F') \textbf{H}_{V_0}(F')\Big),$$
and the  group on the right is a subgroup of $\Big(\Gamma\ltimes \textbf{GL}_{V_+}(F') \Big) \times \Big(\Gamma \ltimes \textbf{Sp}_{V_0}(F') \textbf{H}_{V_0}(F')\Big)$.

 Now $\widetilde{\epsilon'} \circ \det: \Gamma \ltimes \textbf{GL}_{V_+}(F') \longrightarrow \C^{\times}; (\tau, g) \longmapsto  \epsilon'(\det(g))$ is a character of $\Gamma \ltimes \textbf{GL}_{V_+}(F')$. We can take the product  of $\widetilde{\epsilon'} \circ \det$ with  the extended Weil representation $\widetilde{\rho_0'}$ of $\Gamma \ltimes \textbf{Sp}_{V_0}(F') \textbf{H}_{V_0}(F')$ associated to $\psi'$ (or the trivial representation of $\Gamma$ if $\textbf{V}_0=\{ 0\}$), and restrict to $\Gamma \ltimes \Big( \textbf{GL}_{V_+}(F') \times \textbf{Sp}_{V_0}(F') \textbf{H}_{V_0}(F')\Big)$ to get a representation, written $\epsilon' \widetilde{\rho_0'}$.

\begin{proposition}\label{inductionproposition}
The restriction of $\widetilde{\rho'}$ to $\Gamma \ltimes \textbf{P}(F')\textbf{H}_V(F')$ is induced from the representation $\epsilon' \widetilde{\rho_0}$ composed with $\eta'$.
\end{proposition}
\begin{proof}
On restriction to $\textbf{P}(F')\textbf{H}_V(F')$, that is the above mentioned result of G\'erardin. As in \S \ref{OD}, it is enough to compute the trace at $\sigma$. So we need to check that the trace of the induced representation   at $\sigma$  is indeed $q^n$. Now the cosets in $\Gamma \ltimes \textbf{P}(F') \textbf{H}_V(F') /{\Gamma \ltimes \textbf{P}(F') \textbf{H}_{\perp}(F')}$ are represented by $\textbf{H}_V(F')/{\textbf{H}_{\perp}(F')}$ and  for $h\in \textbf{H}_V(F')$, we have
$$(1, h)(\sigma, 1) (1, h^{-1})= (\sigma, h \sigma( h^{-1}))= (\sigma, 1) (1,  \sigma^{-1}(h) h^{-1}),$$
which belongs to $\Gamma \ltimes \textbf{P}(F') \textbf{H}_{\perp}(F')$  only if $h$  is fixed by $\sigma$ modulo $ \textbf{H}_{\perp}(F'))$; but then this means that we can take $h$ to be in $\textbf{H}_V(F)$, in which case, $(1, h)(\sigma, 1) (1, h^{-1})=(\sigma,1)$. All in all, the trace of the induced representation at $\sigma$ is $|\textbf{H}_V(F)/{\textbf{H}_{\perp}(F)}|\cdot q^{n_0}$ with $\dim V_0=2n_0$; since $| \textbf{H}_V(F)/{\textbf{H}_{\perp}(F)}|=|\textbf{V}_+(F)|=q^{n-n_0}$, we get the desired result.
\end{proof}
\section{Reductions}\label{reduction}
We now start the proof of equality ($\star$) in the main theorem. We proceed by induction on $n=\frac{1}{2}\dim_{F}V$. The case where $n=0$ being entirely trivial,  we assume $n> 0$. We fix $i$ and $t$, and as remarked in \S \ref{NM}, we may and do assume that $i$ is prime to $m$, so $d=1$.
We have to prove ($\star$) for a fixed $g$, or equivalently for a fixed norm $h=\nnn_{i,t}(\sigma^i, g)$. \\

If $h$ does not belong to $\textbf{Sp}_{V}(F)\textbf{Z}_{V}(F)$, then $(\sigma^i, g)$ is not conjugate to $\sigma^i \ltimes \textbf{Sp}_{V}(F') Z_V(F')$. In that case, the equality ($\star$) is $0=0$ by the result of \S \ref{SOTCOFTEWR}. So we may assume that  $h$ belongs to $\textbf{Sp}_V(F)\textbf{Z}_V(F)$. But $\textbf{Z}_V(F')$ acts in $\widetilde{\rho'}$ via the character $\psi'$, so on $\Gamma\ltimes \textbf{Z}_V(F')$, the character relation ($\star$) is immediate. Applying  Lemma \ref{lemmaofGyjo} (iii) to the product $\textbf{Sp}_V(F) \textbf{Z}_V(F)$, we see that we may assume that $h$ belongs to $\textbf{Sp}_V(F)$, and $g$ to $\textbf{Sp}_V(F')$.\\

If $h$ stabilizes a non-trivial decomposition $V=V_1 \oplus V_2$ as in \S \ref{OD}, it acts on $\textbf{V}(F)$ via $(h_1, h_2)$ with $h_1 \in \textbf{Sp}_{V_1}(F)$, $h_2 \in \textbf{Sp}_{V_2}(F)$. By  Lemma \ref{lemmaofGyjo} (iii) and  Proposition \ref{orthognoal}, the equality ($\star$) comes from the induction hypothesis applied to $V_1$ and $h_1$ in $\textbf{Sp}_{V_1}(F)$, and  to $V_2$ and $h_1$ in $\textbf{Sp}_{V_2}(F)$.\\

If $h$ stabilizes a non-trivial totally isotropic subspace $V_+$ of $V$, then it belongs to the group $\textbf{P}(F)$ of \S \ref{RTAPPS}, and we can take $g$ in $\textbf{P}(F')$, write $(g_+, g_0)$ for the projection of $g$ to $\textbf{GL}_{V_+}(F') \times \textbf{Sp}_{V_0}(F')$, and similarly $(h_+, h_0)$ for $h$.  We note that
$$\tr\widetilde{\rho_0'}(\sigma^i, g_0)= \tr\rho_0(h_0)$$
  by the induction hypothesis and
  $$\widetilde{\epsilon'}(\det \, g_+)= \epsilon(\det\, h_+)$$
    directly.  The equality ($\star$) for $(g,h)$ then comes from  Proposition \ref{inductionproposition} and Lemma \ref{lemmaofGyjo} applied to the induction from $\textbf{P}\textbf{H}_{\perp}$ to $\textbf{P}\textbf{H}_V$.\\

So the only remaining case  is when $h$ stabilizes no non-trivial orthogonal decomposition $V=V_1 \oplus V_2$, and stabilizes no non-trivial isotropic subspace $V_+$ of $V$. Let us analyze that case. Let $h=su$ be the Jordan decomposition of $h$ into a semi-simple part $s$ and a unipotent part $u$, with $su=us$. Then $F[s]$ is a semi-simple commutative subalgebra of $\End_F(V)$, and the adjoint involution on $\End_F(V)$ associated to the symplectic form on $V$ induces $s \longmapsto s^{-1}$ on $F[s]$.\\

Writing $F[s]$ as a product of fields $(F_{\alpha})_{\alpha \in A}$, we accordingly have a decomposition of $V$ as a direct sum $V= \oplus_{\alpha \in A}V_{\alpha}$, where $F[s]$ acts on $V_{\alpha}$ via $F_{\alpha}$. The involution $s \longmapsto  s^{-1}$ gives a permutation $\alpha \longmapsto  \overline{\alpha}$ on $A$, together with isomorphisms $F_{\alpha} \simeq F_{\overline{\alpha}}$,  and the orthogonal $V_{\alpha}^{\perp}$ of $V_{\alpha}$ is $\oplus_{\beta \neq \alpha} V_{\beta}$. \\

Assume first that $A$ has at least two elements, and take $\alpha$ in $A$. If $\alpha= \overline{\alpha}$ then $V$ is the orthogonal direct sum of $V_{\alpha}$ and $\oplus_{\beta\neq \alpha} V_{\beta}$; each of those subspaces is stable under $s$ and $u$, hence under $h$, which contradicts our assumption on $h$. If $\alpha \neq \overline{\alpha}$, then $h$ stabilizes the non-trivial totally isotropic subspace $V_{\alpha}$, which again contradicts our assumption on $h$.\\

So we see that $A$ has only one element, say $\alpha= \overline{\alpha}$. So $F[s]$ is  a field $E$, and $u$ is an $E$-linear endomorphism of $V$, the involution $s \longmapsto s^{-1}$ on $E$  has a fixed subfield $E_{+}$.

 Assume first that $E=E_+$, i.e. $s=\pm 1$, which implies $E=F$; then $\Ker(u-1_V)$ is a non-trivial subspace of $V$, and any line in that subspace is isotropic and stable under $h$, again a contradiction. We conclude that $E$ is a quadratic extension of $E_+$; then there exists  a skew-hermitian form $\varphi$ on the $E$-vector space $V$---skew-hermitian with respect to $E/{E_+}$ such that, for $v, v'$ in $V$,
$$\langle v, v'\rangle = \tr_{E_+/F}\big( \varphi (v, v')\big).$$
Then $s$ acts on $V$ as an element of $E^{\times}$ with norm $1$ to $E_+$, and $u$ acts as a unipotent element of the unitary group associated to $\varphi$. Now the kernel of $u-1_V$ is orthogonal to its image. If $u \neq 1_V$,  then the intersection $\Im(u-1_V) \cap \Ker(u-1_V)$ is a non-zero isotropic subspace of $V$ stable under $h=su$.\\

 So we conclude that $u=1_V$ and that the $E$-vector space $V$ contains no isotropic non-zero vector with respect to $\varphi$: that implies $V$ has  dimension $1$ over $E$. This very special case will be treated  in the next \S \ref{TCOS} and \S \ref{EOP}.

\section{The case of $SL_2$}\label{TCOS}
We keep the preceding notation, and write $\textbf{U}_{\varphi}$ for the unitary group of $\varphi$ seen as an algebraic group over $E_{+}$ , and $\textbf{T}$ for its restriction of scalars from $E_+$ to $F$. Thus $\textbf{T}$ is a maximally elliptic torus of $\textbf{Sp}_V$ over $F$, and $\textbf{T}(F)=\textbf{U}_{\varphi}(E_+)$  is the group $E^1$  of elements of $E$ with norm $1$ to $E_+$.

Let $\omega$ be the non-trivial character of $\textbf{T}(F)$ of order $2$, and $\omega \psi$ the character of $\textbf{T}(F)\textbf{Z}_V(F)$ given by $\omega$ on $\textbf{T}(F)$ and $\psi$ on $\textbf{Z}_{V}(F)$.

\begin{proposition}\label{sl2}
The virtual representation $\nu=\Ind_{\textbf{H}_V(F)}^{\textbf{T}(F) \textbf{H}_V(F)} (\rho|_{\textbf{H}_V(F)} ) -\Ind_{\textbf{T}(F)\textbf{Z}_V(F)}^{\textbf{T}(F)\textbf{H}_V(F)} \omega\psi$ is the restriction of $\rho$ to $\textbf{T}(F) \textbf{H}_V(F)$.
\end{proposition}
\begin{proof}
The first term of the virtual representation $\nu$ is the sum of the inequivalent irreducible representations $\varphi \rho|_{\textbf{T}(F) \textbf{H}_V(F)}$ where $\varphi$ runs through all characters of $\textbf{T}(F)$. For such a character $\varphi$, the multiplicity of $\varphi \rho|_{\textbf{T}(F) \textbf{H}_V(F)}$ in the second term of the virtual representation is the multiplicity of $\omega$ in $\varphi \rho|_{\textbf{T}(F)}$. But it follows from [\cite{Gero}, p.73] that $\rho|_{\textbf{T}(F)}$ is the direct sum of the characters of $\textbf{T}(F)$ distinct from $\omega$, hence the result.
\end{proof}
The situation we are reduced to is the following: we have an element $s$ of $\textbf{T}(F)$ which is the norm from $F'$ to $F$ of some element $s'$ of $\textbf{T}(F')$( note that $\textbf{T}$ is commutative), and  we want to show that
$$(\star\star ) \quad \quad \tr\widetilde{\rho'}(\sigma^i, s')= \tr\rho(s),$$
for any integer $i$, $0 \leq i \leq m-1$, prime to $m$. Note also that $E=F[s]$ so that in particular $s$ and $s'$ are not $1$.

It is tempting to try and prove it via a proposition similar to the above, but for $\widetilde{\rho'}|_{\Gamma \ltimes \textbf{T}(F') \textbf{H}_V(F')}$. That is not so straightforward, essentially because the $F'$-algebra $E \otimes_FF'$ is generally no longer a field. In this section, we will treat the case $n=1$ so that $E_+=F$; the general case will be dealt with in \S \ref{EOP}.\\

First we assume that $m$ is odd; then $E \otimes_FF'$ is a field $E'$---a quadratic extension of $F'$. We denote by $\omega'$ the order $2$ character of $\textbf{T}(F')$, which is simply $\omega$ composed with the norm from $\textbf{T}(F')$ to $\textbf{T}(F)$, since that norm is surjective.
\begin{proposition}\label{oddn1}
Assume $n=1$ and $m$ odd. Then the virtual representation
$$\nu'= \Ind^{\Gamma \ltimes \textbf{T}(F') \textbf{H}_V(F')}_{\Gamma \ltimes \textbf{H}_V(F')} \big( \widetilde{\rho'}|_{\Gamma\ltimes\textbf{H}_V(F')}\big) - \Ind_{\Gamma \ltimes \textbf{T}(F') \textbf{Z}_V(F')}^{\Gamma \ltimes \textbf{T}(F') \textbf{H}_V(F')} \big( \widetilde{\omega' \psi'}\big) $$
is the restriction of $\widetilde{\rho'}$ to $\Gamma \ltimes \textbf{T}(F') \textbf{H}_V(F')$. Here $\widetilde{\omega' \psi'}$ is the character of $\Gamma \ltimes \textbf{T}(F') \textbf{Z}_V(F')$ obtained by extending $\omega' \psi'$ trivially on $\Gamma$.
\end{proposition}

Let us assume Proposition \ref{oddn1} for a moment, and  prove ($\star \star$) in our special case $n=1$, $m$ odd. As $s$ and $s'$ are not $1$, the first term of the virtual representations contribute nothing to $\tr \widetilde{\rho'} (\sigma^i, s')$ and $\tr \rho(s)$. But it is clear that $\omega \psi$ and $\widetilde{\omega' \psi'}$ verify the Shintani relation for the lifting from $\textbf{T}(F) \textbf{Z}_V(F)$ to $\Gamma \ltimes \textbf{T}(F') \textbf{Z}_V(F')$. By  Lemma \ref{lemmaofGyjo}, it follows that the second terms have equal contribution, which gives ($\star \star$). \\

Let us now  prove Proposition \ref{oddn1}:
we remark that $\nu'$ has positive dimension and that $\tr\nu'(\sigma)= q^m=\tr\widetilde{\rho'}(\sigma)$. The following lemma then shows that $\nu'$ is an irreducible representation. By Proposition \ref{sl2}, with $F'$ as a base field, we see that $\nu'$ is an extension of $\widetilde{\rho'}|_{\textbf{T}(F')\textbf{H}_V(F')}$, and Proposition \ref{oddn1} follows from the equality of traces at $\sigma$.
\begin{lemma}
$\langle \nu', \nu' \rangle=1.$
\end{lemma}
\begin{proof}
By Lemma \ref{lemmaofGyjo},  $\langle \nu', \nu' \rangle=\frac{1}{|\Gal(F'/F) \textbf{T}(F')\textbf{H}(F')|}\sum_{i=0}^{m-1} \sum_{A \in \textbf{T}(F')\textbf{H}(F')} \nu'\Big( (\sigma^i, A)\Big) \overline{\nu'\Big( (\sigma^i, A)\Big)}$

$=\frac{1}{|\Gal(F'/F) \textbf{T}(F')\textbf{H}(F')|} \sum_{i=0}^{m-1}| \sigma^i \ltimes \textbf{T}(F') \textbf{H}(F')| \langle \mathcal{N}_{i, t}( \nu'),  \mathcal{N}_{i, t}( \nu')\rangle$

$=\frac{1}{|\Gal(F'/F) \textbf{T}(F')\textbf{H}(F')|} \sum_{i=0}^{m-1} |\sigma^i \ltimes \textbf{T}(F')\textbf{H}(F')| \langle \mathcal{N}_{i, t}( \nu'), \mathcal{N}_{i, t} (\nu')\rangle$

$=\frac{1}{|\Gal(F'/F) \textbf{T}(F')\textbf{H}(F')|} \sum_{i=0}^{m-1} |\sigma^i \ltimes \textbf{T}(F')\textbf{H}(F')| \langle \rho_{(m,i)}, \rho_{(m,i)}\rangle=1.$
\end{proof}
Now, we assume that $m$ is even, still with $n=1$; then $E \otimes_F F'$ splits as $F' \oplus F'$, so that $\textbf{T}(F')$ is isomorphic to $F'^{\times}$. As before, we let $\omega'$ be the order $2$ character of $\textbf{T}(F')$, which is again  $\omega$ composed with the norm map from $\textbf{T}(F')$  to $\textbf{T}(F)$. Now we write $\eta$ for the order $2$ character of $\Gamma$, and extend as before $\omega' \psi'$ to a character of $\Gamma\ltimes \textbf{T}(F') \textbf{Z}_V(F')$ trivial on $\Gamma$.
\begin{proposition}
Assume $n=1$ and $m$ even. Then the virtual representation
$$\nu'= -\Ind_{\Gamma \ltimes \textbf{H}_V(F')}^{\Gamma \ltimes \textbf{T}(F') \textbf{H}_V(F')} \big( \widetilde{\rho'}|_{\Gamma \ltimes \textbf{H}_V(F')}\big) + \Ind_{\Gamma \times \textbf{T}(F') \textbf{Z}_V(F')}^{\Gamma \ltimes \textbf{T}(F') \textbf{H}_V(F')} \widetilde{\omega' \psi'}$$
is the restriction of $\eta \widetilde{\rho'}$ to $\Gamma \ltimes \textbf{T}(F') \textbf{H}_V(F')$.
\end{proposition}
\begin{proof}
By construction, $\dim \nu'=q^m$ and from Lemma \ref{lemmaofGyjo}, we get $\tr\nu'(\sigma)=-q$. The following lemma, proved as above, then shows that $\nu'$ is irreducible. By Proposition \ref{sl2}, it is an extension of $\widetilde{\rho'}|_{\textbf{T}(F')\textbf{H}_V(F')}$, which has to be  $\eta \widetilde{\rho'}$ since $\tr\nu'(\sigma)=-q$.
\end{proof}
\begin{lemma}
$\langle \nu', \nu'\rangle=1$.
\end{lemma}
We can now prove ($\star \star$) when $n=1$ and $m$ is even. The proof is  as above taking signs in account; the first terms in $\tr
\eta \nu' (\sigma^i, s')$ and $\tr\nu(s)$ contribute nothing, and the second terms are equal, because $\eta(\sigma^i)=-1.$

\section{End of proof}\label{EOP}
We are now ready to prove the formula ($\star \star$) for a general $n \geq 1$ (we keep the notation and assumptions of \S \ref{reduction}). We proceed  by a kind of reduction to \S \ref{TCOS}, and the problem is rather a matter of careful book-keeping.\\

We can see $V$ as a vector space over $E_+$. Endowing it with the form $\delta=\tr_{E/E_+} \varphi$, we get a symplectic vector space of dimension $2$ over $E_+$, for which we write $W$; with it comes a Heisenberg group $\textbf{H}_W$ with $\textbf{H}_W(E_+)= \textbf{W}(E_+) \oplus E_+$ as sets, and there is  an obvious morphism from $\textbf{H}_W(E_+)$ to $\textbf{H}_V(F)$ which is identity on $\textbf{W}(E_+)=\textbf{V}(F)$ and is given by $\tr_{E_+/F}$ on $\textbf{Z}_W(E_+)$. On the other hand $\textbf{Sp}_W(E_+)$ is obviously  a subgroup of $\textbf{Sp}_V(F)$. That gives a morphism $r$ from $\textbf{Sp}_W(E_+)\textbf{H}_W(E_+)$ to $\textbf{Sp}_V(F) \textbf{H}_V(F)$. If we write $\rho_{W}$ for the Weil representation  of $\textbf{Sp}_W(E_+)\textbf{H}_W(E_+)$ associated with $\psi_{E_+}=\psi \circ \tr_{E_+/F}$, we know  [ \cite{Gero},  2.6] that $\rho\circ r$ is isomorphic to $\rho_W$. As $\textbf{T}(F)$ is included  in $\textbf{Sp}_W(E_+)$, we can work with $\rho_W|_{\textbf{T}(F) \textbf{H}_W(E_+)}$ rather than $\rho|_{\textbf{T}(F) \textbf{H}_V(F)}$. Similarly, we want to express $\widetilde{\rho'}|_{\Gamma \ltimes \textbf{T}(F') \textbf{H}_V(F')}$ in terms of Weil representations attached to $2$-dimensional symplectic spaces.\\

Let $e$ be   the   greatest common divisor of  $m$ and $n$;     then $E_+ \otimes_FF'$ splits as the direct sum of $e$ fields $E_+^{\alpha}$, each of  degree $m/e$ over $E_+$ and $n/e$ over $F'$. The group $\Gamma$ permutes the factors transitively, and the stabilizer of each factor is generated by $\sigma^e$; more precisely, $\sigma$ induces an $E_+$-linear isomorphism  of $E_+^{\alpha}$ to $(E_+^{\alpha })^{\sigma}$, and $\sigma^e$ gives  a generator of $\Gal(E_+^{\alpha}/E_+)$ for each $\alpha$. Note however that $\sigma^e$ is not in general the Frobenius automorphism of $E_+^{\alpha}/E_+$; that will not cause any problem with norms, nevertheless, because $\textbf{T}$ is commutative.\\

Now $\textbf{W}(E_+\otimes_FF')$ is endowed with a $E_+ \otimes_F F'$-bilinear  symplectic form( obtained from $\delta$ by scalar extension); it splits as a direct sum of spaces $W_{\alpha}$; each has dimension $2$ over $E_+^{\alpha}$ and carries the $E_{+}^{\alpha}$-bilinear  symplectic form $\delta_{\alpha}$ obtained  from $\delta$ by scalar extension  from $E_+$ to $E_+^{\alpha}$.  The symplectic space $\textbf{W}(E_+ \otimes_F F')$ is the orthogonal direct  sum  of the symplectic subspaces  $W_{\alpha}$. Endowed with  the $F'$-bilinear symplectic form $\tr_{E_+^{\alpha}/F'}(\delta_{\alpha})$, $W_{\alpha}$  is a symplectic vector space $V_{\alpha}$  over $F'$, and  $\textbf{V}(F')$ is isomorphic to  the orthogonal direct sum of the $V_{\alpha}$'s.\\

Now for each $\alpha$, $\textbf{Sp}_{V_{\alpha}}(F')$ is a subgroup of $\textbf{Sp}_V(F')$ and we have a natural inclusion $\textbf{H}_{V_{\alpha}}(F') \longrightarrow \textbf{H}_V(F')$. Altogether, that gives a morphism $\prod_{\alpha}\textbf{Sp}_{V_{\alpha}}(F') \textbf{H}_{V_{\alpha}}(F')  \longrightarrow \textbf{Sp}_{V}(F') \textbf{H}_V(F')$ and it follows from [\cite{Gero}, 4.6] that the inflation of $\rho'$ through that morphism is the tensor product of the Weil representations $\rho_{\alpha}'$ (with respect to $\psi'$).\\

Similarly, through the natural morphism from $\textbf{Sp}_{W_{\alpha}}(E_+^{\alpha}) \textbf{H}_{W_{\alpha}}(E_+^{\alpha})$ to $\textbf{Sp}_{V_{\alpha}}(F')\textbf{H}_{V_{\alpha}}(F')$, $\rho_{\alpha}'$ gives the Weil representation $\rho^+_{\alpha}$  of $\textbf{Sp}_{W_{\alpha}}(E_+^{\alpha}) \textbf{H}_{W_{\alpha}}(E_+^{\alpha})$ attached to $\psi \circ \tr_{E_+^{\alpha}/F'}=\psi_{E_+} \circ \tr_{E_+^{\alpha}/E_+}$.\\

We now want to extend $\otimes_{\alpha} \rho_{\alpha}^+$ to a representation $R$ of  $\Gamma \ltimes \big( \prod_{\alpha} \textbf{Sp}_{W_{\alpha}}(E_+^{\alpha}) \textbf{H}_{W_{\alpha}}(E_+^{\alpha})\big)$ giving trace $q^n$ to $\sigma$. We use tensor induction for that. More precisely, enumerate the $\alpha$'s as $E_+^{\alpha_0},
E_+^{\alpha_1}=[E_+^{\alpha_0}]^{\sigma}, \cdots, E_+^{\alpha_{e-1}}=[E_+^{\alpha_0}]^{\sigma^{e-1}}$(with $E_+^{\alpha^e}=[E_+^{\alpha_0}]^{\sigma^e}$) and fix a model for $\rho_{\alpha_0}^+$  on some space $X$, for example a Schr\"odinger model; extend that model uniquely  to a representation  $\widetilde{\rho^+_{\alpha}}$ of $\langle \sigma^e \rangle \ltimes \textbf{Sp}_{W_{\alpha}}(E_+^{\alpha})\textbf{H}_{W_{\alpha}}(E_+^{\alpha})$, so that $\sigma^e$ has trace $q^n$ (here $\langle \sigma^e \rangle$ denotes  the subgroup of $\Gamma$ generated by $\sigma^e$).  Then there is a unique action of $\Gamma \ltimes \big(\prod_{\alpha} \textbf{Sp}_{W_{\alpha}}(E^{\alpha}_+) \textbf{H}_{W_{\alpha}}(E^{\alpha}_+)\big)$ on $X_0 \otimes X_1 \otimes \cdots \otimes X_{e-1}$, with $X_i= X$ for $i=0, \cdots e-1$, such that
\begin{itemize}
\item[(1)]  For $i=0, \cdots ,  e-1$, $\textbf{Sp}_{W_{\alpha_i}}(E_+^{\alpha_i}) \textbf{H}_{W_{\alpha_i}}(E_+^{\alpha_i})$ acts only on the factor $X_i$ via  $\rho_{\alpha_0}^+\circ \sigma^{-i}$.
\item[(2)]  $\sigma$ acts  by sending  $x_0 \otimes \cdots x_{e-1}$ to $\Bigg( \sigma^e(x_{e-1}) \otimes x_0 \otimes \cdots x_{e-2}\Bigg)$.
\end{itemize}

Clearly  the trace of $\sigma$  on that representation is the trace of  $\sigma^e$  on $X$ i.e. $q^n$, and the restriction  to $\prod_{\alpha} \textbf{Sp}_{W_{\alpha}}(E_+^{\alpha}) \textbf{H}_{W_{\alpha}}(E_+^{\alpha})$ is isomorphic to the product of $\rho_+^{\alpha}$. It follows that if we inflate $\widetilde{\rho'}$ via  the natural homomorphism  from  $\Gamma \ltimes (\prod_{\alpha} \textbf{Sp}_{W_{\alpha}}(E_+^{\alpha}) \textbf{H}_{W_{\alpha}}(E_+^{\alpha}))$ to $\Gamma \ltimes \textbf{Sp}_{V}(F') \textbf{H}_{V}(F')$, then get a representation isomorphic to $R$. \\

Now let us return to our element $s$  in $\textbf{T}(F) =E^1 \subseteq E^{\times}$; we rather see $\textbf{T}$ as a maximally elliptic  torus $\textbf{S}$ of $\textbf{Sp}_W$, so that $s $ is an element of $\textbf{Sp}_W(E_+)$; the symplectic vector space $W_{\alpha_0}$ is obtained from $W$ by scalar extension from $E_+$ to $E_+^{\alpha_0}$, and there is an element $s_0'$ of $\textbf{S}(E_+^{\alpha_0})$ with norm $s$ to $\textbf{S}(E_+)=E^1$. Now consider the element $s'=(s_0', 1, \cdots ,1)$ of $\textbf{S}(E_+ \otimes F')= \prod_{j=0}^{e-1} \textbf{S}(E_+^{\alpha_j})$; then we have $\nnn_{i,t}(\sigma^i, s')=s$ (norm from $\textbf{T}(F') $ to $\textbf{T}(F)$). But from \S \ref{TCOS}, treating the case where $n$ is $1$, we see that
$$\tr \widetilde{\rho_{\alpha_0}^+} (\sigma^{ie}, s'_0)=\tr\widetilde{\rho_{W}}(s).$$
From the construction of $R$, we see that
$$\tr R(\sigma^i, s')= \tr \widetilde{\rho_{\alpha_0}^+}(\sigma^{ie}, s_0'),$$
and  the result follows from \S \ref{TCOS}.
\begin{remark}
Clearly, the considerations of this section  have to do with  the behaviour of Gyoja's norm maps with respect to restriction of scalars. As our concern is more immediate, we have refrained  from developing that aspect along the lines of \cite{Dign2}.
\end{remark}


\begin{thebibliography}{99}
\bibitem[A]{Andr} J.S.ANDRADE,
{\it Repr\'esentations de certains groupes symplectiques},
Bull.Soc.Math.France No. 55-56 (1978).

\bibitem[D1]{Dign} F.DIGNE,
 {\it Shintani descent and L functions of Deligne-Lusztig varieties},
  Proc. of symp. in pure math., 47 (1987) 61--68.

\bibitem[D2]{Dign2} F.DIGNE,
 {\it Descente de Shintani et restriction des scalaires},
  Journal of London Math. Society, vol. 59, No 3 (1999).


\bibitem[DM]{DignM}F. DIGNE, J.MICHEL,
 {\it Representations of finite groups of Lie type},
London Mathematical Society Student Texts, 21. Cambridge University Press, Cambridge, 1991.


\bibitem[Ge]{Gero} P.GERARDIN,
{\it Weil representations associated to finite fields},
J. Algebra 46 (1977) .

\bibitem[Gy]{Gyoja}  A. GYOJA,
{\it  Liftings of irreducible characters of finite reductive groups},
 Osaka J. Math. 16. (1979), no. 1, 130.


\bibitem[MVW]{MVW} C.MOEGLIN, M-F.VIGNERAS, J-L. WALDPURGER,
{\it Correspondances de Howe sur un corps p-adique},
Lecture Notes in Math. Vol 1921, Springer-Verlag, New York, 1987.


\bibitem[Sh]{Shin} K-I. SHINODA,
{\it The Characters of Weil Representations associated to finite
fields}, J. Algebra 66, 251-280 (1980).

\bibitem[S]{Shint} T. SHINTANI,
{\it Two remarks on the irreducible characters of finite general linear groups},
 J. Math. Soc. Japan. 28 (1976), no.2. 396-414.

\bibitem[W]{Weil} A.WEIL,
 {\it  Sur certains groupes d'o\'erateurs unitaires},
 Acta Mathematica 111 (1964), 143-211.
\end{thebibliography}
\end{document}